\documentclass[12pt]{amsart}
\usepackage{amsmath,amsthm,amssymb,bbm,dsfont}
\usepackage{graphicx}
\usepackage{subfig}
\usepackage{cite,float}

\numberwithin{equation}{section}
\newtheorem{theorem}{Theorem}[section]
\newtheorem{lemma}[theorem]{Lemma}
\newtheorem{result}[theorem]{Result}
\newtheorem{corollary}[theorem]{Corollary}
\newtheorem{proposition}[theorem]{Proposition}
\theoremstyle{remark}
\newtheorem{remark}[theorem]{Remark}
\newtheorem{example}[theorem]{Example}
\newtheorem{definition}[theorem]{Definition}

\makeatletter
\@namedef{subjclassname@2010}{%
  \textup{2010} Mathematics Subject Classification}
\makeatother
\begin{document}

\title{a note on squeezing function and its generalizations }

\author[N. Gupta]{Naveen Gupta}
\address{Department of Mathematics, University of Delhi,
Delhi--110 007, India}
\email{ssguptanaveen@gmail.com}

\author[S. Kumar]{Sanjay Kumar Pant}

\address{Department of Mathematics, Deen Dayal Upadhyaya college, University of Delhi,
Delhi--110 078, India}
\email{skpant@ddu.du.ac.in}

\begin{abstract}
This note investigates the relation between  squeezing function and its generalizations. Using 
the relation obtained, we present an alternate method to find expression of generalized squeezing function 
of unit ball corresponding to the generalized complex ellipsoids.
 \end{abstract}

\keywords{squeezing function; holomorphic homogeneous regular domains; 
	quasi balanced domains.}
\subjclass[2010]{32F45, 32H02}
\maketitle

\section{introduction}
This note clarifies the relation between squeezing function and its generalizations, namely, generalized squeezing function and $d$-balanced squeezing function.
 The notion of squeezing function was first introduced by Deng, Guan and Zhang \cite{2012}. Although 
 its genesis can be traced back to the work of Liu et al \cite{Yau2004}, \cite{Yau2005} and of Yeung \cite{yeung}. 
Rong and Yang\cite{gen} introduced the notion of generalized squeezing function. The authors in 
\cite{d-balanced}  extended it to
$d$-balanced squeezing function. Let us see the definitions of all these terms first. 

$\mathbb{B}^n$ denotes unit ball in $\mathbb{C}^n$ and $D\subseteq \mathbb{C}^n$  is used for
bounded domain. The set  of all injective holomorphic 
maps from $D$ to  a domain $\Omega\subseteq \mathbb{C}^n$ is denoted by 
$\mathcal{O}_u(D,\Omega)$.

For $z\in{D}$ 
the squeezing function $S_{D}$ on $D$ is defined as
$$S_{D}(z):=\sup_f\{r:\mathbb{B}^n(0,r)\subseteq f(D), f\in{\mathcal{O}_u(D,\mathbb{B}^n)}\},$$
where $\mathbb{B}^n(0,r)$ denotes ball of radius $r$ centered at the origin.
\smallskip

A domain is said to be holomorphic homogeneous regular if its squeezing function has a positive lower bound. Notion of
holomorphic homogeneous regular (HHR) manifolds was studied by Liu et al \cite{Yau2004}, \cite{Yau2005}. In 2009 \cite{yeung},
Yeung renamed HHR as uniform squeezing property. It was then formally defined in terms of squeezing function by Deng et al in 2012\cite{2012}.
Holomorphic homogeneous regular domains have important significance due to several anlaytic and geometric properties they ensure. 

 In \cite{gen}, Rong and Yang introduced the concept of generalized squeezing function
$S^{\Omega}_D$ for bounded domains $D,\Omega\subseteq \mathbb{C}^n$, where $\Omega$ is a
balanced domain.

Let us quickly see the definitions of a balanced domain and Minkowski function. A 
domain $\Omega\subseteq \mathbb{C}^n $ is said to be balanced if  $\lambda z\in{\Omega}, $
for each $z\in{\Omega}$ and $|\lambda|\leq 1$. 
The Minkowski function denoted by $h_{\Omega}$  on 
$\mathbb{C}^n$ is defined as 
$$h_{\Omega}(z):=\inf \{t>0:\frac{z}{t}\in{\Omega}\}.$$
For $0<r\leq 1$, let $\Omega(r):=\{z\in{\mathbb{C}^n:h_{\Omega}(z)<r\}}$. For 
a bounded domain $D\subseteq \mathbb{C}^n$ and a bounded, balanced, convex domain 
$\Omega\subseteq \mathbb{C}^n$, Rong and Yang introduced the notion of generalized squeezing
function $S_D^{\Omega}$ on $D$ as
\begin{equation*}\label{eqn:gensq}
S^{\Omega}_D(z):=\sup \{r:\Omega(r)\subseteq f(D), f\in{\mathcal{O}_u(D,\Omega)}, f(z)=0\}.
\end{equation*} 

Notion of balanced domains was introduced by Nikolov in his work \cite{nikolov-d}. Let 
$d=(d_1,d_2,\ldots,d_n)\in{\mathbb{Z}_n^{+}},n\geq 2$, we say that a domain $\Omega\subseteq \mathbb{C}^n$
is $d$-balanced if for each $z=(z_1,z_2,\ldots, z_n)\in \Omega$
and $\lambda \in{\overline{\mathbb{D}}}$, $\left(\lambda^{d_1}z_1,\lambda^{d_2}z_2,\ldots, 
\lambda^{d_n}z_n\right) \in \Omega,$ where $\mathbb{D}$ denotes unit disk in $\mathbb{C}$.
It is easy to see that balanced domains are $(1,1,\ldots, 1)$-balanced.

For a $d$-balanced domain $\Omega$, an analogue of Minkowski function,
called $d$-Minkowski function on $\mathbb{C}^n$,
denoted by $h_{d,\Omega}$, is defined as 
$$h_{d,\Omega}(z):=\inf \{t>0:\left(\frac{z_1}{t^{d_1}},\frac{z_2}{t^{d_2}},\ldots,
\frac{z_n}{t^{d_n}}\right)\in{\Omega}\}.$$
For each $0<r\leq 1$, we fix 
$\Omega^d(r):=\{z\in{\mathbb{C}^n:h_{d,\Omega}(z)<r\}}$.

\begin{definition} 
	For a bounded domain $D\subseteq \mathbb{C}^n$, and a bounded, convex, 
	$d=(d_1,d_2,\ldots, d_n)$-balanced domain $\Omega$, 
	\emph{ $d$-balanced squeezing function corresponding to $\Omega$} of the domain $D$,  
	denoted by $S_{d,D}^\Omega$(also called the\emph{ $d$-balanced squeezing function} for brevity)
	is given by:
	\begin{equation*}\label{eqn:gensq}
	S_{d,D}^\Omega(z):=\sup \{r:\Omega^d(r)\subseteq f(D), f\in{\mathcal{O}_u(D,\Omega)}, f(z)=0\}.
	\end{equation*}
\end{definition}

Similar to the case of squeezing function, holomorphic homogeneous regular domains are the ones whose generalized squeezing function
(or $d$-balanced squeezing function) has a positive lower bound.

As we know that squeezing function was introduced in the context of uniform squeezing property (HHR), so 
it is natural to expect that regularity remains unaltered when one shifts from 
squeezing function to its generaliations.
More precisely, it is natural to expect that if a domain $D$ is regular in 
$S_D$, then it should be regular in $S_D^{\Omega}$ or $S_{d,D}^{\Omega}$. To address this we analyze relations between these 
objects in Theorems \ref{prop:rel_sq_gensq} and \ref{prop:sq_dbal}.
 We would like to 
point out that apart from addressing the question about regularity of a domain, these results will also serve its purpose in estimating squeezing function
for different domains. Another similar question related to regularity is behaviour of these generalizations corresponding to different model domains.
What we mean by this is: For a given domain $D$ and two bounded, balanced, convex domains $\Omega_1, \Omega_2$, if $D$ is regular as per the setup of 
$S_D^{\Omega_1}$, will it be regular as per $S_D^{\Omega_2}$? A similar question can be formed for $S_{d,D}^{\Omega}.$ 
One can not be unjust for expecting this and this is the content of our Corollary \ref{thm:rel_gensq} 
and Corollary \ref{thm:dbal}.

\section{$S_D$ and $S_D^{\Omega}$}

In \cite[~Lemma 1.2]{npoly} a relation between the squeezing function and the squeezing function corresponding 
to polydisk is presented. Such a relation immediately deduces the equivalence between holomorphic homogeneous 
regular domains in these two setups and is helpful in obtaining estimate for squeezing function [See examples in 
\cite{npoly}]. We present a similar relation between squeezing function and generalized squeezing function, which 
in particular will imply that the holomorphic homogeneous regular domains are the same for $S_D$ and $S_D^{\Omega}$. 
We point out here that equivalence of holomorprhic homogeneous regular domains is established in 
\cite[~Theorem 3.9]{gen}, but the inequality 
established here is more useful in estimating generalized squeezing function.  
For a domain  $D\subseteq \mathbb{C}^n$ and $z_1,z_2\in D$, the Carathéodory pseudodistance  $c_{D}$ on $D$ is defined as 
$$c_D(z_1,z_2)=\sup_f \{p(0,\mu):f\in{\mathcal{O}(\mathbb{D}, D), f(z_1)=0, f(z_2)=\mu}\},$$ where $p$ denotes the Poincar\'e metric on unit disc $\mathbb{D}$. We take $c_D^*=\tanh c_D$. Using \cite[~Lemma 3.2]{rong}, for a bounded, balanced, convex domain $D$, we have $c_D^*(0,z)=h_D(z)$ for $z\in D$.

\begin{proposition}{\label{prop:rel_sq_gensq}}
	Let $D\subseteq \mathbb{C}^n$ be bounded and $\Omega\subseteq \mathbb{C}^n$ be bounded, balanced, convex domain.
	Then for $a\in D$ the following holds:
	\begin{enumerate}
		\item $S_D(a)\geq \dfrac{\alpha}{R}S_D^{\Omega}(a),$
		\item $S_D^{\Omega}(a)\geq \dfrac{\alpha}{R}S_D(a),$	
	\end{enumerate}
where $\alpha= dist(0,\partial \Omega)$ and $R=\dfrac 1 2 diam(\Omega)$.
\end{proposition}

\begin{proof} $\mbox{ }$
\begin{enumerate} 
\item 
	For $a\in D$,  let $f:D\to \mathbb{B}^n$ be an injective holomorphic map with $f(a)=0$. Let $r>0$ 
	be such that $B^n(0,r)\subseteq f(D)$. Consider $g:D\to \Omega$ defined as $$g(z):=\alpha f(z).$$
	 Note that $g$ 
	is injective holomorphic with $g(a)=0$ and $g(z)\in{\Omega}$ for all $z\in D$. Also note that 
	$B^n(0,\alpha r)
	\subseteq g(D)$. We claim that $\Omega\left(\dfrac{\alpha r}{R}\right)\subseteq g(D).$ 
	Before proving our claim, we make an
	 observation that $\Omega\subseteq B^n(0,R)$ and therefore 
	\begin{equation}\label{eq:sub}
	 c^*_{B^n(0,R)}(0,z)\leq c^*_{\Omega}(0,z) =h_{\Omega}(z)\ \mbox{for all}\ z\in \Omega.	 
\end{equation}
In order to prove our claim, see that for $z\in{\Omega\left(\dfrac{\alpha r}{R}\right)},$
using equation \eqref{eq:sub}, we have $$c^*_{B^n(0,R)}(0,z)\leq c^*_{\Omega}(0,z) =h_{\Omega}(z)<\dfrac{\alpha r}{R}.$$ 
This gives us $z\in{\mathbb{B}^n(0,\alpha r)} $ and thus $z\in g(D).$ Therefore, we get our 
claim and hence 
	$S_D^{\Omega}(a)\geq \dfrac{\alpha r}{R}$, which implies $r\leq \dfrac{R}{\alpha} S_D^{\Omega}(a)$. 
	This establishes (1).

	\item Let $f:D\to \Omega$ be an injective holomorphic map with $f(a)=0$. Let $r>0$ be such that $\Omega(r) 
	\subseteq f(D)$. Consider $g:D\to \mathbb{B}^n$ defined as $$g(z):=\dfrac{f(z)}{R}.$$ Observe that $g$ is injective 
	holomorphic with $g(a)=0$. Also, $\Omega\left(\dfrac {r}{R}\right)=\dfrac 1 R \Omega(r)\subseteq g(D)\subseteq \mathbb{B}^n.$ Since 
	$\mathbb{B}^n(0,\alpha)\subseteq \Omega$, therefore 
	\begin{equation}\label{eq:sub1}
	h_{\Omega}(z)=c^*_{\Omega}(0,z) \leq c^*_{B^n(0,\alpha)}(0,z) \ \mbox{for all} \ z\in \mathbb{B}^n(0,\alpha).
	\end{equation}
	We claim that $\mathbb{B}^n\left(0,\dfrac {\alpha r}{R}\right)\subseteq g(D).$ 
	For $z\in \mathbb{B}^n\left(0,\dfrac {\alpha r}{R}\right),$ using
	 equation \eqref{eq:sub1} (note that $r/R<1$), we have
	 $$	h_{\Omega}(z)=c^*_{\Omega}(0,z) \leq c^*_{B^n(0,\alpha)}(0,z)=\dfrac {1}{\alpha} |z|<\dfrac r R.$$
	\end{enumerate}
Thus we get $z\in \Omega\left(\dfrac {r}{R}\right)\subseteq g(D)$, which proves our claim and 
therefore $S_D(a)\geq \dfrac{\alpha r}{R}$.
This establishes (2).
	
\end{proof}

\begin{remark}
		The two inequalities in the above theorem can not be attained simultaneously unless $\Omega =
		\mathbb{B}^n$. To see this, note that the two inequalities are attained simultaneously only when $
		\alpha= R$.
\end{remark}
Note that previous theorem in particular implies that holomorphic homogeneous regular domains in 
the set up of $S_D$ and $S_D^{\Omega}$ are the same.

The interaction between generalized squeezing function for different model domains $\Omega$ is presented 
in the following easy 
corollary to the above proposition.
For relation between  $S_D^{\Omega}$ for different model domains, where $\Omega$ is {\emph homogeneous} as 
well, one can see \cite[~Theorem 4.4]{gen}. We will see analogue of this result for $d$-balanced squeezing function 
in Theorem \ref{thm:analogue}.

\begin{corollary}\label{thm:rel_gensq}
	Let $D\subseteq \mathbb{C}^n$ be a bounded domain and $\Omega_1, \Omega_2\subseteq \mathbb{C}^n $ be bounded, balanced, 
	covex domains, then the following holds:
	
	\begin{enumerate}
		\item $S_D^{\Omega_2}(a)\geq \dfrac{\alpha_1\alpha_2}{R_1R_2}S_D^{\Omega_1}(a),$
		\item $S_D^{\Omega_1}(a)\geq \dfrac{\alpha_1\alpha_2}{R_1R_2}S_D^{\Omega_2}(a),$	
	\end{enumerate}
where $\alpha_i= dist(0,\partial \Omega_i)$ and $R_i= \dfrac 1 2 diam(\Omega_i),\ i=1,2$.
	\begin{proof}
		To prove (1), use part (2) of Proposition \ref{prop:rel_sq_gensq} for $\Omega_2$ and then use part
		 (1) for $\Omega_1$. One can prove (2) in similar way.
	\end{proof}
	
\end{corollary}

A simple consequence of the above stated corollary is:
 $S_D^{\Omega_1}$ has a positive lower bound if and only if  $S_D^{\Omega_2}$
has positive lower bound. We would like to mention here, this statement has been proved in \cite[~Theorem 3.9 ]{gen}
 separately, but the bounds obtained here are the simplest possible and are useful in its own right.

\section{$S_D$ and $S_{d,D}^{\Omega}$}

Let us recall few properties of d-Minkowski function and $d$-balanced domains.
\begin{result}\label{res:basicminko}
	For a $d$-balanced domain $\Omega\subseteq \mathbb{C}^n$, the following holds:
	\rm{( see\cite[~Remark 2.2.14]{pflug})}
	\begin{enumerate}
		\item $\Omega =\{z\in{\mathbb{C}^n}:h_{d,\Omega}(z)<1\}$.
		\item $h_{d,\Omega}\left(\lambda^{d_1}z_1,\lambda^{d_2}z_2,\ldots, \lambda^{d_n}
		z_n\right)=|\lambda |h_{d,\Omega}(z)$ for each $z=(z_1,z_2,\ldots, z_n)\in{\mathbb{C}^n}$
		and $\lambda\in{\mathbb{C}}.$
		\item $h_{d,\Omega}$ is upper semicontinuous.
	\end{enumerate}
\end{result}
For a bounded, $d$-balanced, convex domain $\Omega$, using \cite[~Theorem 1.6]{bharali}, we have 
$$\tanh^{-1}h_{d,\Omega}(z)^L\leq c_{\Omega}(0,z)=k_{\Omega}(0,z)\leq \tanh^{-1}h_{d,\Omega}(z),$$
where $L=\max_{1\leq i\leq n}d_i.$
Similar to the case of generalized squeezing function, the following theorem for 
$d$-balanced squeezing function holds.

\begin{proposition}\label{prop:sq_dbal}
		Let $D\subseteq \mathbb{C}^n$ be bounded and $\Omega\subseteq \mathbb{C}^n$ be bounded, 
		$d$-balanced, convex domain.
	Then for $a\in D$ the following holds:
	\begin{enumerate}
		\item $S_D(a)\geq \dfrac{\alpha}{P}S_{d,D}^{\Omega}(a)^L,$
		\item $S_{d,D}^{\Omega}(a)\geq \dfrac{\alpha}{R}S_D(a),$	
	\end{enumerate}
	where $\alpha= dist(0,\partial \Omega)$ and $R=\dfrac 1 2 diam(\Omega), P=R+1, L=\max_{1\leq i\leq n}d_i.$
\end{proposition}
\begin{proof}
Let $a\in D$ and $f:D\to \Omega$ be injective holomorphic map with $f(a)=0$. Let $r>0$ be such that 
$\Omega^d(r)\subseteq f(D)\subseteq \Omega=\Omega^d(1).$ Consider $g:D\to \mathbb{C}^n$ defined as 
$$g(z):=\left(\dfrac{f_1(z)}{P^{d_1}},\dfrac{f_2(z)}{P^{d_2}}, \ldots, 
\dfrac{f_n(z)}{P^{d_n}}\right).$$
It is obvious that $g$ is injective holomorphic with $g(a)=0$. 
Also using Result \ref{res:basicminko}(2) and the fact that $\Omega\subseteq \mathbb{B}^n(0,P)$
 it is easy to infer that 
$g(D)\subseteq \mathbb{B}^n $.  Therefore, $g:D\to \mathbb{B}^n$. Next we claim that 
$\Omega^d\left(\dfrac{\alpha r^L}{P}\right)\subseteq g(D)$. 
Let $z\in \Omega^d\left(\dfrac{ \alpha r^L}{P}\right) $, that is,
$h_{d,\Omega}(z)<\dfrac{ \alpha r^L}{P}<1$. Therefore, we get $z\in \Omega$. 
Also using the decreasing property of Caratheodory metric, 
we have $c^*_{\mathbb{B}^n(0,P)}(0,z)\leq c^*_{\Omega}(0,z)$ for $z\in {\Omega}$. 
Thus we have $\dfrac{1}{P}|z|\leq c^*_{\Omega}(0,z)\leq h_{d,\Omega}(z)<
\dfrac{ \alpha r^L}{P}$. Thus, $|z|<
\alpha r^L\leq \alpha$ and for $z\in{\mathbb{B}^n(0,\alpha)}$, using the decreasing property, 
we have $h_{d,\Omega}(z)^L\leq c^*_{\Omega}(0,z)\leq c^*_{\mathbb{B}^n(0,\alpha)}(0,z)=\dfrac{1}{\alpha}|z|.$
Therefore, we obtain $h_{d,\Omega}(z)<\dfrac{\alpha r}{P}$. Now using Result \ref{res:basicminko}(2), 
we get  $z\in g(D)$. 
Therefore, $\Omega^d\left(\dfrac{\alpha r^L}{P}\right)\subseteq g(D)$ and this finally yeilds (1). 
Proceding on the similar lines, one can prove (2). 	
\end{proof}

The following corollary ensures that the holomorohic homogeneous regular domains are same 
corresponding to different $d$-balanced domains.

\begin{corollary}\label{thm:dbal}
	Let $D\subseteq \mathbb{C}^n$ be bounded and $\Omega_1,\Omega_2\subseteq \mathbb{C}^n$ be
	 bounded, convex domains. Assume that $\Omega_1$ is $d$-balanced, and $\Omega_2$ is 
	 $d'$-balanced, then for $a\in D$, the following holds:
	\begin{enumerate}
		\item $S_{d',D}^{\Omega_2}(a)\geq \dfrac{\alpha_1\alpha_2S_{d,D}^{\Omega_1}(a)^L}{P_1R_2},$
		\item $S_{d,D}^{\Omega_1}(a)\geq \dfrac{\alpha_1\alpha_2S_{d',D}^{\Omega_2}(a)^{L'}}{R_1P_2},$
	\end{enumerate}
where $ P_i=R_i+1, \alpha_i= dist(0,\partial \Omega_i)$,
 $R_i=\dfrac 1 2 diam(\Omega_i),\ \ i=1,2$; $L=\max_{1\leq i\leq n}d_i$ and $L'=\max_{1\leq i\leq n}d_i'$.
\end{corollary}
\begin{proof}
	One can prove this by proceeding on the same lines as in the Corollary \ref{thm:rel_gensq}.
\end{proof}

Similarly to the Theorem \cite[~Theorem 4.4]{rong} the following theorem gives relation between 
$d$-balanced squeezing function for different model domains, 
when domains in consideration are bounded, $d$-balanced, convex and \emph{homogeneous}.

\begin{theorem}\label{thm:analogue}
let $\Omega_1, \Omega_2\subseteq \mathbb{C}^n$ be bounded, convex, homogeneous domains and $D\subseteq 
\mathbb{C}^n$ be a bounded domain. Assume that 
$\Omega_1$ is $d$-balanced and $\Omega_2$ is $d'$-balanced. Then 
\begin{enumerate}
	\item $S_{d',D}^{\Omega_2}(z)\geq S_{d,\Omega_1}^{\Omega_2}(0)S_{d,D}^{\Omega_1}(z)^L,$
	\item $S_{d,D}^{\Omega_1}(z)\geq S_{d',\Omega_2}^{\Omega_1}(0)S_{d',D}^{\Omega_2}(z)^{L'},$
\end{enumerate}	
where $L=\max_{1\leq i\leq n}d_i$ and $L'=\max_{1\leq i\leq n}d_i'$.
\end{theorem}

We will need the following lemma to prove this theorem. 

\begin{lemma}\label{lem:analogue}
	Let $\Omega\subseteq \mathbb{C}^n$ be a $d$-balanced domain and $\Omega_1=\Omega^d(r),$ where 
	$0<r\leq 1$. Then $$h_{d,\Omega_1}=\dfrac 1 rh_{d,\Omega}.$$ 
\end{lemma}
\begin{proof} First of all, observe that $\Omega_1$ is bounded, $d$-balanced and $\Omega_1 \subseteq 
	\Omega '$.
Let $A=\{t>0:\left(\frac{z_1}{t^{d_1}},\frac{z_2}{t^{d_2}},\ldots, \frac{z_n}{t^{d_n}} \right)\in
{\Omega}\}$ and $A_1=\{t>0:\left(\frac{z_1}{t^{d_1}},\frac{z_2}{t^{d_2}},\ldots, \frac{z_n}{t^{d_n}} \right)\in
{\Omega '}\}$.	It is easy to see that $t\in A$ if and only if $\dfrac t r\in A_1$. Therefore we get 
$A_1=\dfrac 1 r A$ and hence the result follows.
\end{proof}

\begin{proof}[Proof of theorem \ref{thm:analogue}]
Let $z\in D$ and $f:D\to \Omega_1$ be injective holomorpphic with $f(z)=0$. Let $r>0$ be such that 
\begin{equation}\label{eqn:analouge1}
\Omega_1^d(r)\subseteq f(D).
\end{equation}	

Let $\alpha=S_{d,\Omega_1}^{\Omega_2}(0)$, then using Theorem \cite[~Theorem 4.6 ]{d-balanced}, 
there exists  an injective holomorphic map $g:\Omega_1 \to \Omega_2$  with $g(0)=0$ such that
\begin{equation}\label{eqn:analogue2}  
 \Omega_2^{d'}(\alpha)\subseteq g(\Omega_1).
 \end{equation}	
 Consider $F:D\to \Omega_2$ defined as $F=g\circ f$, then $F$ is injective holomorphic with $F(z)=0$.
 Let us fix $\Omega=\Omega_2^{d'}(\alpha)$.
 We will first prove that $B_{\Omega}^l\left(0,\beta\right)\subseteq g\left(\Omega_1^d(r)\right),$ 
 where $\beta=\tanh ^{-1}r^L$ and $l$ denotes Lempert's function.

 Let $z\in B_{\Omega}^l\left(0,\beta\right) $, that is, $l_{\Omega}(0,z)<\beta$. Since $\Omega
 \subseteq g(\Omega_1)$, therefore $z=g(w)$ for some $w\in \Omega$ and using decreasing property of 
 Lempert's function, we get  $l_{g(\Omega_1)}(0,g(w))<\beta$, thus we get $l_{\Omega_1}(0,w)<\beta$. Also 
 using \cite[~Theorem 1.6]{bharali}, we have $\tanh^{-1}h_{d,\Omega_1}(z)\leq l_{\Omega_1}(0,w)$. Thus we 
 obtain $h_{d,\Omega_1}(w)<r$ and therefore $z=g(w)\in {g\left(\Omega_1^d(r)\right)}.$
 
 So we have obtained $B_{\Omega}^l\left(0,\beta\right)\subseteq g\left(\Omega_1^d(r)\right)$ and therefore 
 using Equation \eqref{eqn:analouge1}, we have  
 \begin{equation}\label{eqn:analogue3}
 B_{\Omega}^l\left(0,\beta\right)\subseteq g(f(D))=F(D). 
 \end{equation}
 Next we claim to prove $\Omega_2^{d'}(\alpha r^{L})\subseteq 
 B_{\Omega}^l\left(0,\beta\right)$. For $z\in{\Omega_2^{d'}(\alpha r^{L})}$, $h_{d',\Omega_2}(z)<
 \alpha r^{L}$. Using  Lemma \ref{lem:analogue}, we get $h_{d',\Omega}(z)<r^{L}$. Now observing that 
 the right hand side of inequality in  \cite[~Theorem 1.6]{bharali} is true without pseudoconvexity assumption 
 of the domain, we get $\tanh l_{\Omega}(0,z)\leq h_{d',\Omega}(z)<r^{L}$. This implies $z
 \in{B_{\Omega}^l\left(0,\beta\right)}$. This proves our claim and thus using Equation 
 \eqref{eqn:analogue3}, we get 
 $$\Omega_2^{d'}(\alpha r^{L})\subseteq F(D).$$
 Hence we get $S_{d',D}^{\Omega_2}(z)\geq\alpha r^L$ and this completes the proof of (1). 
 Similarly one can prove (2).
\end{proof}

\section{generalized squeezing function for complex ellipsoid}
A \emph{complex ellipsoid} in $\mathbb{C}^n$ is a domain of the form 
$$E(m)=\{z\in \mathbb{C}^n:\sum_{j\in{I_1}}|z_j|^{2m_j}<1,\, |z_k|<1 \, \forall k\in{I_2}\},$$
where $1\leq m_i\leq \infty $, $I_1=\{1\leq j\leq n: m_j\neq \infty\}$ and $I_2=\{1,2,\ldots, n\}\setminus 
I_1$. $E(m)$ is convex if and only if $m_1,m_2,\ldots, m_n\geq \frac 1 2$. $E(m)=\mathbb{B}^n,$ when 
$m=(1,1, \ldots, 1)$ and $E(m)=\mathbb{D}^n$, when $m_j=\infty $ for all $1\leq j\leq n$.

A \emph{generalized complex ellipsoid} denoted by $E(p,m)$ with $p=(p_1,p_2,\ldots, p_k),\, 
m=(m_1,m_2,\ldots, m_k)$, $p_1+p_2+\ldots +p_k=n$ is given by 
$$E(p,m)=\{z\in \mathbb{C}^{p_1}\times \mathbb{C}^{p_2}\times\ldots \times  \mathbb{C}^{p_k} 
:\sum_{j=1}^k\|z_j\|^{2m_j}<1\}.$$
If any of $m_j=\infty$, then $E(p,m)$ is understood as in the case of $E(m).$ In particular, 
when $m_j=\infty $ for each $j=1,2,\ldots, k$, $$E(p,\infty)=\mathbb{B}^{p_1}\times
 \mathbb{B}^{p_2}\times \ldots \times \mathbb{B}^{p_k}.$$
Note that $E(p,\infty)$ is bounded, balanced and convex. Also, it is easy to see that $\Omega$ is 
homoegeneous. 

Let us denote by $\Omega_i=\mathbb{B}^{p_i}\, , i=1,2,\ldots, k$, and $D=\mathbb{B}^n,$ 
then $\Omega=E(p,\infty)=\Omega_1\times 
\Omega_2\times \ldots \times \Omega_k$. Therefore, using \cite[~Corollary 4.3 ]{rong}, we get that
$$S_D^{\Omega}(z)=S^D_{\Omega}(w),$$ for each $z\in D$ and $w\in \Omega$. It is readily seen using 
\cite[~Theorem 7.5]{2012} that  
$$S^D_{\Omega}(w)=\left(S_{\Omega_1}(w_1)^{-2}+S_{\Omega_2}(w_2)^{-2}+\ldots+ S_{\Omega_k}(w_k)^{-2}\right)^{-1/2}
=\dfrac{1}{\sqrt{k}}$$
 and therefore $S_D^{\Omega}(z)=\dfrac{1}{\sqrt{k}}$.
We will present here an alternate method to obtain this expression using the inequality obtained in 
Proposition \ref{prop:rel_sq_gensq}. We will require the following observation for this.

 It is easy to see that $\alpha=dist(0,\partial \Omega)=1$ and 
$R=\frac 1 2 diam (\Omega)=\sqrt{k}$. We also need the following lemma, whose proof is based on the 
technique used by H. Alexander in the proof of \cite[~Proposition 2]{alexander}.

\begin{lemma}\label{lem:E}
	Let $f:D\to \Omega$ be injective holomorphic with $f(0)=0$. Let $r>0$ be such that 
	$\Omega(r)\subseteq f(D)$, then $r\leq \dfrac{1}{\sqrt{k}}.$
\end{lemma}
\begin{proof}
Let us consider a function $g:\Omega(r)\to D$ defined as $$g(z):=f^{-1}(z).$$
For any $\epsilon>0$, consider $\mathbb{B}^n\left(0,\dfrac{1+\epsilon}{\sqrt{k}}
\right)$. Note that using \cite[~Theorem 7.5]{2012}, it follows that  
$S_{\Omega(r)}^D=\dfrac{1}{\sqrt{k}}$, therefore $g(\Omega(r))$ does not contain 
$\mathbb{B}^n\left(0,\dfrac{1+\epsilon}{\sqrt{k}}\right)$. Let $a\in{\mathbb{B}^n
	\left(0,\dfrac{1+\epsilon}{\sqrt{k}}\right)}$ be such that $b=f(a)\notin{\Omega(r)}=
\mathbb{B}^{p_1}(r)\times \mathbb{B}^{p_2}(r)\times \ldots \times 
\mathbb{B}^{p_k}(r)$. This yields one of $b_1,b_2,\ldots ,b_k$, say $b_1$, has modulus 
$\geq r$. Therefore,  using Schwarz's lemma, we get $$r\leq \|b_1\|=\|f_1(a)\|\leq \|a\|\leq 
\dfrac{1+\epsilon}{\sqrt{k}},$$ which completes the proof.
	
\end{proof}

\begin{example}
For $D=\mathbb{B}^n$, $\Omega=E(p,\infty)$, $$S_D^{\Omega}(z)=\frac{1}{\sqrt{k}},$$
for $z\in D.$ 
\end{example}
Before we prove this, observe that when $p_i=1$ for all $i=1,2,\ldots, k$, then $\Omega=\mathbb{D}^n$. For 
this case, we have $S_D^{\Omega}(z)=\dfrac{1}{\sqrt{n}},$ \cite[~Example 1]{npoly}.
\begin{proof}
Using Proposition \ref{prop:rel_sq_gensq}, we know that 
$$S_D^{\Omega}(a)\geq \dfrac{\alpha}{R}S_D(a)=\dfrac{\alpha}{R},$$
where $\alpha$ and $R$ as in the proposition.	Therefore, we get
 $$S_D^{\Omega}(z)\geq \dfrac{1}{\sqrt{k}}.$$
On the other hand, using Lemma \ref{lem:E}, we get 
$$S_D^{\Omega}(z)\geq \dfrac{1}{\sqrt{k}}.$$

\end{proof}

\section*{Acknowledgement}
The first author is thankful to Peter Pflug for fruitful discussions for some parts of this note.



\medskip



\begin{thebibliography}{00}
	\smallskip




\bibitem{alexander} H. Alexander,  Extremal holomorphic imbeddings between the ball 
and polydisc, {\em Proc. Amer. Math. Soc.}, {\bf 68}(2) (1978), 200--202.

\bibitem{bharali} G. Bharali, Non-isotropically balanced domains, Lempert function estimates,
and the spectral Nevanlinna-Pick theorem, arXiv:0601107.

\bibitem{2012} F. Deng, Q. Guan, L. Zhang, Some properties of squeezing functions
on bounded domains,
{\em Pacific Journal of Mathematics}, {\bf 57}(2) (2012), 319--342.	

\bibitem{deng2019} F. Deng, X. Zhang, Fridman's invariants, squeezing functions and exhausting domains,
{\em Acta Math. Sin. (Engl. Ser.)}, {\bf 35}(2019), 1723--1728.








\bibitem{npoly} N. Gupta, S. K. Pant, Squeezing function corresponding to polydisk, {\em Complex Anal. Synerg.} {\bf 8}, 12 (2022), DOI:10.1007/s40627-022-00100-8.
\bibitem{d-balanced} N. Gupta, S. K. Pant, $d$-balanced squeezing function, {\em Complex Var. Elliptic Equ.},  DOI:10.1080/17476933.2021.2007380.


	\bibitem{pflug} M. Jarnicki, P. Pflug, Invariant Distances and Metrics in Complex Analysis,
2nd edition(extended), De Gruyter Expositions in mathematics, Berlin, 2013.


\bibitem{Yau2004} K. Liu, X. Sun,  S. T. Yau, Canonical metrics on the moduli 
space of Riemann surfaces, I, {\em J. Differential Geom.}, {\bf 68}(3) (2004), 571--637.


\bibitem{Yau2005} K. Liu, X. Sun, and S.T. Yau, Canonical metrics on the moduli space of Riemann
surfaces, II, {\em J. Differential Geom.}, {\bf 69}(1) (2005), 163--216.






\bibitem{nikolov-d} N. Nikolov, The symmetrized polydisc cannot be exhausted by domains 
biholomorphic to convex domains, {\em Ann. Polon. Math.} {\bf 88} (2006), 279–283.



\bibitem{gen} F. Rong, S. Yang, On Fridman invariants and generalized squeezing functions, {\em Chin. Ann. Math. Ser. B}, {\bf 43}, 161–174 (2022), DOI:10.1007/s11401-022-0320-y.

\bibitem{rong} F. Rong, S. Yang, On the comparison of the Fridman invariant and the 
squeezing function, {\em Complex Var. Elliptic Equ.}, DOI:10.1080/17476933.2020.1851210.


\bibitem{yeung} S. K. Yeung, Geometry of domains with the uniform squeezing property,
{\em  Adv. Math.}, {\bf 221}(2) (2009), 547--569.


\end{thebibliography}
\end{document}